\documentclass[12pt, twoside]{article}
\usepackage{amsmath,amsthm,amssymb}
\usepackage{times}
\usepackage{enumerate}
\usepackage{mathtools}

\def\titlerunning#1{\gdef\titrun{#1}}
\makeatletter
\def\author#1{\gdef\autrun{\def\and{\unskip, }#1}\gdef\@author{#1}}

\makeatother

\newtheorem{thm}{Theorem}[section]

\newtheorem*{theorem-non}{Theorem}

\newtheorem{lem}[thm]{Lemma}
\newtheorem{pro}[thm]{Proposition}

\newtheorem{nota}[thm]{Notation}

\theoremstyle{plain}
\newtheorem{obs}[thm]{Observation}
\newtheorem{ques}[thm]{Question}

\usepackage{graphicx}
\newcommand{\coverpage}[3]{\thispagestyle{empty}
    \addtocounter{page}{-1}
\null\vspace*{-1cm} \hfill\includegraphics[scale=0.2]{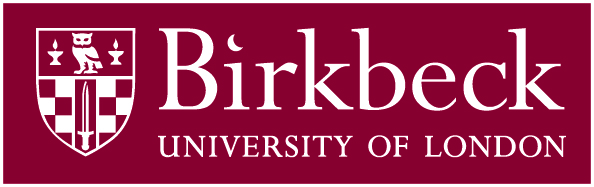} \vskip 2in
\begin{center} \begin{minipage}{0.7\textwidth}\begin{center}\Huge\bf{#1}\end{center} \end{minipage}\end{center}  \vfill
\begin{center} {\large By}\bigskip\\ {\large #2}\\ \end{center} \vfill 
 \framebox{\begin{minipage}{\textwidth}
Birkbeck Pure Mathematics Preprint Series\hfill
Preprint Number #3 \\ \\
\null\hfill www.bbk.ac.uk/ems/research/pure/preprints 
\end{minipage}}
\newpage}

\begin{document}


\baselineskip=17pt


\coverpage{A note on filled groups}{Chimere S. Anabanti and Sarah B. Hart}{17}
\titlerunning{A note on filled groups}

\title{A note on filled groups}
\author{Chimere S. Anabanti \thanks{The first author is supported by a Birkbeck PhD Scholarship}\\ c.anabanti@mail.bbk.ac.uk \and Sarah B. Hart\\ s.hart@bbk.ac.uk}
\date{}

\maketitle

\begin{abstract}
\noindent Let $G$ be a finite group and $S$ a subset of $G$. Then $S$ is {\em product-free} if $S \cap SS = \emptyset$, and $S$ {\em fills} $G$ if $G^{\ast} \subseteq S \cup SS$. A product-free set is locally maximal if it is not contained in a strictly larger product-free set.  Street and Whitehead [J. Combin. Theory Ser. A \textbf{17} (1974), 219--226] defined a group $G$ as {\em filled} if every locally maximal product-free set in $G$ fills $G$. Street and Whitehead classified all abelian filled groups, and conjectured that the finite dihedral group of order $2n$ is not filled when $n=6k+1$ ($k\geq 1$). The conjecture was disproved by the current authors in [Austral. Journal of Combinatorics \textbf{63 (3)} (2015), 385--398], where we also classified the filled groups of odd order. This brief note completes the classification of filled dihedral groups and discusses filled groups of order up to 100.
\end{abstract}

\section{Preliminaries}
Let $S$ be a non-empty subset of a group $G$. We say $S$ is product-free if $S\cap SS=\varnothing$. A product-free set $S$ is locally maximal if whenever $T$ is product-free in $G$ and $S\subseteq T$, then $S=T$. Finally a product-free set $S$ fills $G$ if $G^*\subseteq  S \cup SS$ (where $G^*$ is the set of all non-identity elements of $G$), and $G$ is a filled group if every locally maximal product-free set in $G$ fills $G$.  This definition, due to Street and Whitehead, was motivated by the observation that a product-free set in an elementary abelian 2-group $A$ is locally maximal if and only if it fills $A$, and hence the elementary abelian 2-groups are filled groups.  We list some results that we shall use in the next two sections.

\begin{thm}\label{T1}\begin{enumerate} 
\item[(i)] \cite[Lemma 3.1]{GH2009} Let $S$ be a product-free set in a finite group $G$. Then $S$ is locally maximal if and only if $G=T(S) \cup \sqrt{S}$, where $T(S):=S \cup SS \cup SS^{-1} \cup S^{-1}S$ and $\sqrt{S}:=\{x \in G:x^2 \in S\}$.
\item [(ii)] \cite[Lemma 1]{SW1974} If $G$ is a filled group, and $N$ is a normal subgroup of $G$, then $G/N$ is filled.
\item [(iii)]\cite[Theorem 2]{SW1974} A finite abelian group is filled if and only if it is $C_3$, $C_5$ or an elementary abelian 2-group.
\item [(iv)] \cite[Lemma 2.3]{AH2015} $C_3$ is the only filled group with a normal subgroup of index $3$.
\item [(v)] \cite[Lemma 2.5]{AH2015} If $G$ is a filled group with a normal subgroup $N$ of index $5$ such that not every element of order $5$ is contained in $N$, then $G \cong C_5$.
\item [(vi)]\cite[Theorem 2.6]{AH2015} The only filled groups of odd order are $C_3$ and $C_5$.
\item [(vii)] \cite[Prop 2.8]{AH2015} For $n \geq 2$, the generalized quaternion group of order $4n$ is not filled.
\end{enumerate}\end{thm}

\section{Classification of filled dihedral groups}
A list of non-abelian filled groups of order less than or equal to $32$ was given in \cite{AH2015}. There are eight such groups: six are dihedral, and the remaining two are 2-groups. The dihedral groups on the list are those of order 6, 8, 10, 12, 14 and 22. Our aim in this section is to show that these are in fact the only filled dihedral groups. 

\begin{nota}
We write $D_{2n}=\langle x,y|~x^{n}=y^2=1, xy=yx^{-1} \rangle$ for the dihedral group of order $2n$ (where $n>2$). In $D_{2n}$, the elements of $\langle x\rangle$ are called rotations and the elements of $\langle x\rangle y$ are called reflections. For any subset $S$ of $D_{2n}$, we write $A(S)$ for $S \cap \langle x \rangle$, the set of rotations of $S$, and $B(S)$ for $S \cap \langle x \rangle y$, the set of reflections of $S$.
\end{nota}

\begin{obs}\label{obs1}
Suppose $S$ is a subset of $D_{2n}$. Let $A = A(S)$ and $B = B(S)$. Then,  because of the relations in the dihedral group, we have  $AA^{-1} = A^{-1}A$,  $AB = BA^{-1}$ and $B^{-1} = B$. Therefore \begin{align*}
SS &= AA \cup BB \cup AB \cup BA;\\
SS^{-1} &= AA^{-1} \cup BB \cup AB;\\
S^{-1}S &= AA^{-1} \cup BB \cup BA;\\
T(S) &= A \cup B \cup AA \cup AA^{-1} \cup BB \cup AB \cup BA \\ &= S \cup SS \cup AA^{-1}.
\end{align*}
We also  note that $\sqrt S = \sqrt A \subseteq \langle x \rangle$. 
\end{obs}

\begin{pro}\label{P1}
Let $n$ be an odd integer, with $n \geq 13$. Then $D_{2n}$ is not filled.
\end{pro}
\begin{proof} Let $n$ be an odd integer. Then there is an odd number $k$ for which $n$ is either $5k-6$, $5k-4$, $5k-2$, $5k$ or $5k+2$.  \\

 Suppose first that $n$ is $5k-2$ for an odd integer $k$. Since $n \geq 13$, we note that $k \geq 3$. Now consider the following set $S$. $$S:=\{x^k,x^{k+2}, \cdots, x^{3k-2}; y, xy, \cdots, x^{k-1}y \}.$$ 
 We calculate that  
 $A(SS)=\{x^{2k},x^{2k+2},\cdots,x^{5k-3} \} \cup \{1,x,\cdots,x^{k-1}\} \cup \{x^{4k-1},x^{4k},\cdots,1\}$ and $B(SS)=\langle x \rangle y - B(S)$. Observe that $x^{3k}\notin S \cup SS$; so $S$ does not fill $G$.\\
 Let $A=A(S)$. Then $AA^{-1}=\{1,x^2,x^4,\cdots,x^{2k-2}\} \cup \{x^{3k},x^{3k+2}, \cdots,x^{5k-4},1\}$. Thus $T(S)=G$. By Theorem \ref{T1}(i) therefore, $S$ is locally-maximal product free in $G$, but we have noted that $S$ does not fill $G$.
 
 Next we suppose $n = 5k$ for an odd integer $k$, and again since $n \geq 13$, we have $k \geq 3$. Taking the same set $S = \{x^k,x^{k+2}, \cdots, x^{3k-2}; y, xy, \cdots, x^{k-1}y \}$ we find that $S$ is locally maximal product-free but does not fill $G$.
 
 Now suppose $n=5k+2$ for $k\geq 3$ and odd. The set $T$ given by $$T=\{x^{k-2},x^k, \cdots, x^{3k-2}; y, xy, \cdots, x^{k-3}y \}$$ is locally maximal product-free in $G$ (again using Theorem \ref{T1}(i)), but does not fill $G$ since $x^{3k}\notin T \cup TT$.
 
Next suppose $n=5k-6$ for $k\geq 5$ and odd. Then the set $U$ given by $$U=\{x^k, x^{k+2}\cdots, x^{3k-2}; y, xy, \cdots, x^{k-1}y \}$$ is a locally maximal product-free set in $G$ that does not fill $G$. 

Finally, consider the case $n=5k-4$ for $k\geq 5$ and odd. The set $V$ given by $$V=\{x^{k-2},x^k, \cdots, x^{3k-4}; y, xy, \cdots, x^{k-3}y \}$$ is locally maximal product-free set in $G$ which does not fill $G$.
\end{proof}

\begin{lem}
\label{lem} The only filled dihedral groups of order up to 32 are $D_6$, $D_8$, $D_{10}$, $D_{12}$, $D_{14}$ and $D_{22}$. Furthermore, the dihedral group of order 44 is not filled.
\end{lem}

\begin{proof}
The filled groups of order up to 32 were classified in \cite{AH2015}, and this gives the information about filled dihedral groups in that range. For the second part of the lemma, an application of Theorem \ref{T1}(i), along with a straightforward calculation, shows that the set $\{x^2,x^5,x^8,x^{18}, x^{21},x^5y, x^{16}y\}$ is locally maximal product-free in  $D_{44}$, but it does not fill $D_{44}$. 
\end{proof}

\begin{thm}\label{C1}
The only filled dihedral groups are $D_{6}$, $D_{8}$, $D_{10}$, $D_{12}$, $D_{14}$ and $D_{22}$.
\end{thm}
\begin{proof}
Proposition \ref{P1} along with Lemma \ref{lem} deal with the case of dihedral groups of order $2n$ where $n$ is odd. The only such groups that are filled are $D_{6}$, $D_{10}$, $D_{14}$ and $D_{22}$. It remains to consider dihedral groups of order $4m$, for an integer $m$. The groups $D_{8}$ and $D_{12}$ are filled, so assume $m \geq 4$ and we proceed by induction on $m$. Certainly $D_{16}$, $D_{20}$, $D_{24}$ and $D_{28}$ are not filled, so we can assume $m \geq 8$. If $G$ is dihedral of order $4m$ with $m \geq 8$, then the quotient of $G$ by its centre is dihedral of order $2m$. If $G$ is filled then by Theorem \ref{T1}(ii), this group of order $2m$ must be filled. If $m$ is odd, then by Proposition \ref{P1} and Lemma \ref{lem} and the assumption that $m \geq 8$, we have $m = 11$, so that $G$ is $D_{44}$, which by Lemma \ref{lem} is not filled. Suppose $m$ is even, so $m=2t$ for some $t$ with $4 \leq t < m$. Then inductively $D_{4t}$ is not filled, so $G$ is not filled. This completes the proof.
\end{proof}

\section{Filled groups of orders up to 100}
This section focuses on the classification of filled groups of orders up to $100$. The filled groups of order up to 32 are already known. Above order 32, we can immediately discount abelian, dihedral and generalized quaternion groups. There is also the following observation, which deals with groups of order $4p$, where $p$ is an odd prime.  

\begin{lem}\label{P31}
Let $p$ be an odd prime. If $G$ is a filled group of order $4p$, then $p=3$ and the group is $D_{12}$. \end{lem}
\begin{proof}
The only filled group of order 12 is $D_{12}$. So we may assume $p \geq 5$. Recall that if $p\not\equiv 1 \mod~4$, then $D_{4p}$ and $Q_{4p}$ are the only non-abelian groups of order $4p$, and for $p\equiv 1 ~mod~4$, the non-abelian groups of order $4p$ are $D_{4p}$, $Q_{4p}$ and $G_{4p}$, where $$Q_{4p}:= \langle x, y|~ x^{2p} = 1, y^2 = x^p, xy = yx^{-1}\rangle$$ and $$G_{4p}:=C_p \rtimes_{\mu} C_4$$ for $C_p=\langle a \rangle$, $C_4=\langle b \rangle$ and $\mu(b)(a) = a^{(p-1)/2}$. In the light of Theorem \ref{T1}(vii) and Theorem \ref{C1}, we only need to show that $G_{4p}$ is not filled. But $G_{4p}$ has a quotient isomorphic to $C_4$; hence, by Theorem \ref{T1}(ii), $G_{4p}$ is not filled.
\end{proof}

We next use the fact that filled groups must have filled quotients (Theorem \ref{T1}(ii)) to eliminate nearly all other candidates. If any candidate has a normal subgroup of index 3, then by Theorem \ref{T1}(iv) it cannot be filled (in spite of the fact that $C_3$ is filled). The remaining groups have the property that all their quotients are filled. These are dealt with using a brute-force search for locally maximal product-free sets that do not fill the group in question. We do this using the program below in \cite{GAP4}.



\begin{verbatim}
### A program to test whether a given set is product-free
PFTest:= function(T)
    local  x, y, prod;
    for x in T do for y in T do
        if x*y in T then
           return 1;
        fi;
    od; od;
    return 0;
end;

### This program stops as soon as a locally maximal product-free
set which does not fill the group G is found.
nonfill:=function(G,k)
local L, lis, counterex, combs, x, pf, H, y, z, i, q, g;
L:=List(G); lis:=[];
for i in [2..Length(L)] do
  Add(lis,L[i]);
od;
# Take L to be the lis of non-id elements of G
L:=lis; counterex:=[];
for pf in IteratorOfCombinations(L,k) do
    if PFTest(pf)=0 then
       H:=Difference(L,pf);
       for y in [1..k] do for z in [1..k] do
           H:=Difference(H, [pf[y]*pf[z]]);
       od; od;
       if Size(H) > 0 then
          for y in [1..k] do for z in [1..k] do
            H:=Difference(H, [pf[y]*(pf[z])^-1, ((pf[y])^-1)*pf[z]]);
          od; od;
          for q in L do
               if q^2 in pf then
                  H:=Difference(H, [q]);
               fi;
          od;
          if Size(H) = 0 then
             counterex:=pf; break;
          fi; 
       fi;
    fi;
od;
return counterex;
end;
\end{verbatim}

\noindent Our investigation shows that if $k$ is the order of a filled group with $32 < k \leq 127$, then $k=64$. Certainly the elementary abelian group of order 64 is filled. There are two further groups of order 64 whose status we were not able to determine. So, there are at most three filled groups of order 64. This means that every known filled group is either abelian ($C_3$, $C_5$ or an elementary abelian 2-group), dihedral ($D_6, D_8$, $D_{10}$, $D_{12}$, $D_{14}$ or $D_{22}$) or a 2-group, with the only  known nonabelian examples being $D_8$, $D_8 \times C_2$ and $D_8 \ast Q_8$ (one of the two extraspecial groups of order 32).     
 We conclude with the following question:
\begin{ques}
Are there any filled groups of order greater than 22 which are not 2-groups?
\end{ques}

\end{document}